\documentclass[11pt,twoside, a4paper, english, reqno]{amsart}
\usepackage[dvips]{epsfig}
\usepackage{amscd}
\usepackage{amssymb}
\usepackage{amsthm}
\usepackage{amsmath}
\usepackage{latexsym}
\usepackage{amsaddr}

\usepackage{upref}
\usepackage{hyperref}

\hypersetup{linkcolor=blue, colorlinks=true
,citecolor = red}
\usepackage{color}
\theoremstyle{plain}
\newtheorem{thm}{Theorem}[section]
\newtheorem{thma}{Theorem}
 
\theoremstyle{plain}
\newtheorem{lem}[thm]{Lemma}

\theoremstyle{definition}
\newtheorem{defi}{Definition}[section]

\newcommand{\bn}{\mathbb{B}^{N}}
\newcommand{\rn}{\mathbb{R}^{N}}

\newcommand{\rnn}{\mathbb{R}^{2}}
\newcommand{\hn}{\mathbb{H}^{N}}

\newcommand{\Cbn}{C^{\infty}_c(\mathbb{B}^N)}
\newcommand{\Chn}{C^{\infty}_c(\mathbb{H}^N)}
\newcommand{\authorfootnotes}{\renewcommand\thefootnote{\@fnsymbol\c@footnote}}%

\numberwithin{equation}{section} \allowdisplaybreaks
\begin{document}
\title[Adams' inequality with exact growth in $\mathbb{H}^4$ and Lions lemma]
{Adams' inequality with exact growth in the hyperbolic space $\mathbb{H}^4$ and Lions lemma}

 \author{Debabrata Karmakar $^\dagger$}
\address{$^\dagger$ TIFR  Centre for Applicable Mathematics, Bangalore 560065, India}
 \thanks{$^\dagger$ TIFR-CAM, Post Bag No. 6503
Sharadanagar, Yelahanka New Town, Bangalore 560065.\\ Email: debkar@math.tifrbng.res.in.}

\begin{abstract}
  In this article we prove Adams inequality with exact growth condition in the four dimensional hyperbolic space 
  $\mathbb{H}^4,$
  \begin{align}
   \int_{\mathbb{H}^4} \frac{e^{32 \pi^2 u^2} - 1}{(1 + |u|)^2} \ dv_g \leq C ||u||^2_{L^2({\mathbb{H}^4})}.
   \end{align}
  $\mbox{for all} \ u \in C^{\infty}_c(\mathbb{H}^4) \ \mbox{with} \ \int_{\mathbb{H}^4} (P_2 u) u \ dv_g \leq 1.$ 

 \vspace{0.2 cm}
 
 We will also establish an Adachi-Tanaka type inequality in this settings. Another aspect of this article is the P.L.Lions 
 lemma in the hyperbolic space. We prove P.L.Lions lemma for the Moser functional and for a few cases of the Adams functional
 on the whole hyperbolic space.
\end{abstract}
\maketitle

\textbf{Keywords :} Hyperbolic spaces , Adams inequalities , Exact growth condition , Lions lemma.

\vspace{0.15 cm}

\textbf{Mathematics Subject Classification (2010) :} 26D10 , 46E35 , 46E30.

 \section{Introduction} 
 Let $\Omega$ be a bounded smooth domain in $\rn,$ and $k$ be a positive integer less than $N$ then 
 the limiting case of Sobolev embedding asserts that : when $kp = N,$  $W^{k,p}_0(\Omega)$ is continuously embedded in 
 $L^q(\Omega)$ for all $1 \leq q < + \infty,$ whereas $W^{k,p}_0(\Omega) \nsubseteq L^{\infty}(\Omega).$
 
 To answer the question of optimal embedding in the case $k = 1,$ Pohoza\'{e}v (\cite{Pohozaev}), Trudinger (\cite{Trudinger}) 
 and later  Moser (\cite{Moser}) proved, what is now popularly known as 
 Trudinger-Moser inequality. They established the following sharp inequality:
 \begin{align} \label{MT ineq}
  \sup_{u \in W^{1,N}_0(\Omega), \int_{\Omega} |\nabla u|^N \leq 1} \int_{\Omega} 
  e^{\beta |u|^{\frac{N}{N -1}}} \ dx < + \infty, \ \ \ \ \mbox{iff} \ \ \ \beta \leq \alpha_N,
 \end{align}\textit{}
 where $\alpha_N = N\omega_{N - 1}^{\frac{1}{N - 1}}$ and $\omega_{N - 1}$ is the surface measure of the 
 $(N - 1)$-dimensional unit sphere. The sharp constant $\alpha_N$ in \eqref{MT ineq} is due to Moser (\cite{Moser}).
 
 \vspace{0.2 cm}
 
 Trudinger-Moser inequality had its consequences in the study of PDE's with exponential nonlinearity, especially 
 the ones coming from geometry and physics. As a consequence, there has been many developments in this area since 
 1971. On one hand people were studying its validity in various domains (not necessarily of finite measures) and
 other geometries, many attentions had also been devoted to the possible generalizations and 
improvements.
  
  One interesting generalization in the whole $\rn$ is due to Adachi-Tanaka (\cite{AdachiTanaka}). 
 While studying the inequality
 \eqref{MT ineq} on $\rn,$ they understood, if one work with the gradient norm then it is not possible to get Trudinger-Moser
 inequality in its original form. The inequality they proved looks slightly different from \eqref{MT ineq}.
 They established the following inequality :
 
 \vspace{0.3 cm}
 
 Let $\Phi_{1,N}(t) := e^t - \sum_{j = 1}^{N - 2} \frac{t^j}{j\!},$ then for all $\beta \in (0,\alpha_N)$
 there exists $C(\beta) > 0$ such that :
 \begin{align} \label{Adachi-Tanaka}
  \int_{\rn} \Phi_{1,N}(\beta |u|^{\frac{N}{N - 1}}) \ dx \leq C(\beta) ||u||^N_{L^N(\rn)} \ \ \  \mbox{for all} \ 
  u \in W^{1,N}(\rn), \int_{\rn} |\nabla u|^N \leq 1,  
 \end{align}
and this is sharp in the sense that \eqref{Adachi-Tanaka} does not hold for any $\beta \geq \alpha_N.$
However, \eqref{MT ineq} type inequality does hold if we replace gradient norm by the full $W^{1,N}(\rn)$ norm
(see \cite{Ruf}, \cite{LiRuf}) and interestingly it can be derived from the inequality \eqref{Adachi-Tanaka}
(see \cite{MaSani}). 

In this context in dimension two S. Ibrahim-N.Masmoudi-K.Nakanishi(\cite{IMaN}) showed that one can achieve  
the best constant $4\pi$ in \eqref{Adachi-Tanaka} by weakening the exponential nonlinearity slightly. 
The same inequality was extended later by Masmoudi-Sani(\cite{MaSani}) for any $N \geq 2.$ They obtained the 
following exact growth condition for \eqref{Adachi-Tanaka} type inequality to hold with best constant $\alpha_N.$
\begin{thma}[\cite{MaSani}]
 Let $N \geq 2$ then there exists a constant $C_N>0$ such that
 \begin{align}
  \int_{\rn} \frac{\Phi_{1,N}(\alpha_N |u|^{\frac{N}{N - 1}})}{(1 + |u|)^{\frac{N}{N - 1}}} \leq C_N ||u||^N_{L^N(\rn)},
  \ \mbox{for all} \ u \in W^{1,N}(\rn), \int_{\rn} |\nabla u|^N \leq 1.
 \end{align}
Moreover, this inequality fails if the power $\frac{N}{N - 1}$ in the denominator is replaced by any $p < \frac{N}{N - 1}.$
\end{thma}

A great deal of literature is available on Trudinger-Moser inequality on the whole $\rn.$
For works related to Trudinger-Moser inequality on $\rn$ we refer to
\cite{Cao}, \cite{Panda},\cite{DoO}, \cite{Ruf}, \cite{LiRuf} and the references therein. Extensions 
to compact Reimannian manifolds are dealt with in \cite{Li1}, \cite{Li2}. See also \cite{AdiDruet}, \cite{AdiSandeep}
\cite{AdiYang}, \cite{Yang}
for many other variants of the inequality \eqref{MT ineq}.

\vspace{0.3 cm}

The study of Trudinger-Moser inequality on bounded domains with infinite measure is also a centre of 
attraction for the past few years.
One such example is the study of Trudinger-Moser inequality on the hyperbolic space. 
In fact Trudinger-Moser type inequality on the hyperbolic space was first studied by Mancini-Sandeep(\cite{MS}),
Adimurthi-Tinterev(\cite{AdiT}) in dimension two and Lu-Tang(\cite{LuT}) for dimension $N >2.$ They obtained 
sharp inequality in this setting with the same best constant as the one in the Euclidean space.

Let $\hn$ be the hyperbolic $N$-space with $N \geq 2$ then
\begin{align}
 \sup_{u \in \Chn, \int_{\hn} |\nabla_g u|_g^N \leq 1} \int_{\hn} \Phi_{1,N}(\beta |u|^{\frac{N}{N - 1}}) \ dv_g < + \infty,
 \ \ \ \mbox{for all} \ \beta \leq \alpha_N,
\end{align}
where $dv_g$ is the volume element and $\nabla_g$ is the hyperbolic gradient (see section 2 for definitions).

A comprehensive study of Trudinger-Moser inequality with exact growth condition and Adachi-Tanaka type inequality 
in the hyperbolic space were performed by Lu-Tang (\cite{LuTa}). They established the following sharp inequality: 

\begin{thma}[\cite{LuTa}]
 For any $u \in W^{1,N}(\hn)$ satisfying $||\nabla_g u||_{L^N(\hn)} \leq 1,$ there exists a constant 
 $C(N) > 0$ such that
 \begin{align}
  \int_{\hn} \frac{\Phi_{1,N}(\alpha_N |u|^{\frac{N}{N - 1}})}{(1 + |u|)^{\frac{N}{N - 1}}} \ dv_g 
  \leq C(N) ||u||^N_{L^N(\hn)}.
 \end{align}
\end{thma}
There are several articles on the Trudinger-Moser inequality on the hyperbolic space. We refer to
\cite{LuT}, \cite{MS}, \cite{MST}, \cite{Tin}, \cite{YZ} for related works on the hyperbolic space.

\vspace{0.3 cm}

Another interesting generalization of Trudinger-Moser inequality is due to Adams(\cite{A}) about the validity of 
\eqref{MT ineq} type inequality for higher order derivatives. 
He established exponential integrability of the functions belonging to the Sobolev space $W^{k,p}_0(\Omega)$ where $kp = N.$
He also found the best constant $\beta_0(k,N)$ in this case which matches with the best constant in Trudinger-Moser
inequality when $k = 1.$ His result can be formulated as follows : 

\begin{thma}[\cite{A}] \label{Adams}
Let $\Omega$ be a bounded domain in $\rn$ and $k < N.$ Then there exists a constant $C=C(k,N) > 0$ such that 
 \begin{align} \label{Adamsintegral}
 \sup_{u \in C^{k}_c(\Omega), \int_{\Omega} |\nabla^k u|^p \leq 1}\int_{\Omega} e^{\beta |u(x)|^{p^{\prime}}} \ dx \leq C |\Omega|,
\end{align}
for all $\beta \leq \beta_0(k,N),$ where $p = \frac{N}{k},p^{\prime} = \frac{p}{p - 1} $,
\begin{align} \label{beta0}
 \beta_0(k,N) = 
 \begin{cases}
  \frac{N}{\omega_{N-1}} \left[\frac{\pi^{\frac{N}{2}}2^k \Gamma\left(\frac{k + 1}{2}\right)}
  {\Gamma \left(\frac{N - k + 1}{2}\right)}\right]^{p^{\prime}}, \ \ \ \mbox{if $k$ is odd}, \\ \ \
  \frac{N}{\omega_{N-1}} \left[\frac{\pi^{\frac{N}{2}}2^k \Gamma\left(\frac{k}{2}\right)}
  {\Gamma \left(\frac{N - k }{2}\right)}\right]^{p^{\prime} }, \ \ \ \mbox{if $k$ is even}, \\
  \end{cases}
\end{align}
and $ \nabla^k $ is defined by
 \begin{align}\label{higra}
  \nabla^k :=
  \begin{cases}
   \Delta^{\frac{k}{2}} , \ \ \ \ \ \ \ \mbox{if} \ k \ \mbox{is even}, \\
   \nabla \Delta^{\frac{k-1}{2}} , \ \ \mbox{if} \ k \ \mbox{is odd}.
  \end{cases}
\end{align}
Furthermore, if $\beta > \beta_0,$ then the supremum in \eqref{Adamsintegral} is infinite.
\end{thma}

Like Trudinger-Moser inequality, Adams inequality also attracted many improvements and generalizations. 
Extensions of \eqref{Adamsintegral} to the whole $\rn$
was studied by Ruf-Sani (\cite{RufSani}), Lam-Lu (\cite{LamLu}). Tarsi (\cite{tar}) showed Adams type inequality holds 
for more general class of functions which contains $W^{k, \frac{N}{k}}_0(\Omega)$ as a closed subspace. 
A complete generalizations of Adams inequality on compact Riemannian manifolds is due to Fontana(\cite{Fonta}). 
See also \cite{FntanaM} for Adams type inequality on arbitrary  measure spaces. 
Recently Fontana(\cite{Fontana15}) proved Adams inequality on the hyperbolic space. 

Another interesting aspect of Adams inequalities is to look for exponential integrability under the constrained 
which governs critical $Q_{\frac{N}{2}}$ curvature. In fact in \cite{SD}  Adams type inequality
on the hyperbolic space were considered under the condition $\int_{\hn} (P_k u)u \ dv_g \leq 1$, where 
$P_k$ is the $2k$-th order GJMS operator (see section 2).  They established the following sharp inequality :
\begin{thma}[\cite{SD}]\label{HYADA}
Let $\hn$ be the $N$ dimensional hyperbolic space with $N$ even and $k =\frac N2$ then,
 \begin{align}
  \sup_{u \in \Chn , ||u||_{k,g} \leq 1} \ \int_{\hn} \left(e^{\beta u^2} - 1\right) \ dv_g < +\infty
 \end{align}
iff $\beta \leq \beta_0(k,N)$, where $\beta_0(k,N)$ is as defined in \eqref{beta0} and $||u||_{k,g}$ is the norm defined by
\begin{align} 
 ||u||_{k,g} := \left[\int_{\hn} (P_k u)u \ dv_g\right]^{\frac 12}, \ \mbox{for all} \ \ u \in \Chn,
\end{align}
where $P_k$ is the $2k$-th order GJMS operator on the hyperbolic space $\hn$.
\end{thma}

Coming back to the discussion in $\rn,$ one can ask : does Adachi-Tanaka type inequality  holds for 
higher order derivatives, or what is the exact growth
condition for the Adams functional in $\rn?$ 
In fact very few results is known  in this 
direction. In $\mathbb{R}^4$ the question about exact growth have been answered by Masmoudi and Sani(\cite{MSani}).

\begin{thma} [\cite{MSani}]\label{MasoudiSani1}
 There exists a constant $C>0$ such that
 \begin{align}
  \int_{\mathbb{R}^4} \frac{e^{32\pi^2u^2} - 1}{(1 + |u|)^2} \ dx \leq C||u||^2_{L^2(\mathbb{R}^4)},
 \end{align}
for all  $u \in W^{2,2}(\mathbb{R}^4)$ with $||\Delta u||_{L^2(\mathbb{R}^4)} \leq 1.$
Moreover, this fails if the power $2$ in the denominator is replaced with any $p < 2.$
\end{thma}
They also proved  Adachi-Tanaka type inequality in this setting.
\begin{thma} [\cite{MSani}]\label{MasoudiSani2}
 For any $\alpha \in (0, 32\pi^2),$ there exists a constant $C(\alpha) > 0$ such that
 \begin{align}
  \int_{\mathbb{R}^4} \left(e^{\alpha u^2} - 1\right) \ dx \leq C(\alpha)||u||^2_{L^2(\mathbb{R}^4)},
 \end{align}
for all  $u \in W^{2,2}(\mathbb{R}^4)$ with $||\Delta u||_{L^2(\mathbb{R}^4)} \leq 1,$ and this inequality
fails for any $\alpha \geq 32\pi^2.$
\end{thma}
 
One of the goal of this article is to address these two types of inequalities in the hyperbolic space. We 
 will establish the exact growth condition (Theorem \ref{exact growth theorem 1}) and also 
Adachi-Tanaka type inequality (Theorem \ref{exact growth theorem 2})  for functions belonging to $H^2(\mathbb{H}^4)$
(see main results for precise statement).

 Another aspect of the Trudinger-Moser or Adams inequalities is the concentration compactness phenomena. P.L.Lions in his 
celebrated paper \cite{P.L.L} proved concentration compactness alternatives for the Moser functional.
Among many other results he proved the following : 
 
 If a sequence $\Big\{u_m : ||\nabla u_m||_{L^N(\Omega)} = 1 \Big\}$ in $W^{1,N}_0(\Omega)$ converges weakly to $u$ 
 in $W^{1,N}_0(\Omega)$ with $u \not\equiv 0,$ then
 \begin{align} \label{P.L.Lions lemma}
  \sup_m \int_{\Omega} e^{\alpha_N p |u_m|^{\frac{N}{N - 1}}} \ dx < + \infty,
  \ \ \ \mbox{for all} \ p < \left(1 - ||\nabla u||_{L^N}^N\right)^{-\frac{1}{N - 1}}.
   \end{align}
 
The inequality  \eqref{P.L.Lions lemma} does not give any extra information than 
the Trudinger-Moser inequality if the sequence converges weakly to zero, but the implication of the above lemma
is that the critical Moser functional is compact outside a weak neighbourhood of zero. We refer \cite{CCH}
for a detailed discussions on the P.L.Lions lemma and their generalizations to the functions with
unrestricted boundary condition.

The concentration compactness alternatives for Adams functional has been carried out by M. do \'{O}-Macedo
(\cite{DoMac}). Recently P.L.Lions lemma  for the Moser functional has been extended on whole $\rn$ 
by M. do. \'{O} et. al.(\cite{DoOetal}). To the best of our knowledge, for higher order derivatives
P.L.Lions type lemma on domain with infinite measure  is still an open question except for few cases (\cite{YYang1}).

 Another goal of this article is to cast P.L.Lions type lemma on the hyperbolic space. We prove P.L.Lions type lemma 
for two different settings for the Adams functional (see Theorem \ref{main pll 2} and Theorem \ref{main pll 3}). 
We also established analogous results for the Moser functional on the whole $\hn$ (Theorem \ref{main pll1}).

The followings are the main results of this article:

\pagebreak 
\begin{center}
 \textbf{Main Results}
\end{center}

The first two results are concerning the exact growth condition in $H^2(\mathbb{H}^4).$
 \begin{thm} \label{exact growth theorem 1}
There exists a constant $C>0$ such that for all $u \in C^{\infty}_c(\mathbb{H}^4)$ with \\
$\int_{\mathbb{H}^4} P_2(u)u \ dv_g \leq 1,$ there holds,
\begin{align} \label{Adams inequality with exact growth}
   \int_{\mathbb{H}^4} \frac{e^{32 \pi^2 u^2} - 1}{(1 + |u|)^2} \ dv_g \leq C ||u||^2_{L^2({\mathbb{H}^4})}.
\end{align}
Moreover, this is optimal in the sense that if we consider
\begin{align}
   \int_{\mathbb{H}^4} \frac{e^{\beta u^2} - 1}{(1 + |u|)^p} \ dv_g ,
\end{align}
then \eqref{Adams inequality with exact growth} fails to hold either for $\beta > 32 \pi^2$ and $p = 2$ or,
$\beta = 32\pi^2$ and $p < 2.$
\end{thm}

\begin{thm} \label{exact growth theorem 2}
 For any $\alpha \in (0 , 32\pi^2)$ there exists a constant $C(\alpha) > 0$ such that for all  
 $u \in C^{\infty}_c(\mathbb{H}^4)$ with $\int_{\mathbb{H}^4} P_2(u)u \ dv_g \leq 1,$ there holds,
 \begin{align}
  \int_{\mathbb{H}^4} \left (e^{\alpha u^2} - 1\right) \ dv_g \leq C(\alpha) ||u||^2_{L^2({\mathbb{H}^4})},
 \end{align}
and the inequality fails for any $\alpha \geq 32\pi^2.$
\end{thm}

The next three results in this article are concerning P.L.Lions lemma in the whole hyperbolic space.   
\begin{thm} \label{main pll1}
 Let $\{u_m : ||\nabla u_m||_{L^N(\hn)} = 1\}$ be a sequence in $W^{1,N}(\hn)$ such that $u_m$ converges weakly to a nonzero
 function $u$ in $W^{1,N}(\hn).$ Then  
 
 \begin{align}
  \sup_m \int_{\hn} \Phi_{1,N} (\alpha_N p |u_m|^{\frac{N}{N - 1}}) \ dv_g < + \infty,
 \end{align}
for all  $p < \left(1 - ||\nabla u||_{L^N(\hn)}^N\right)^{-\frac{1}{N - 1}}.$
\end{thm}

\begin{thm} \label{main pll 2}
Let $\{u_m : ||u_m||_{k,g} = 1\}$ be a sequence in $H^k(\hn)$ such that $u_m$ converges weakly to a nonzero function 
 $u$ in $H^k(\hn).$ Then for all $p < \left(1 - ||u||^2_{k,g}\right)^{-1},$ there holds
 \begin{align}
  \sup_m \int_{\hn} \left(e^{\beta_0p u^2_m} - 1\right) \ dv_g < +\infty,
 \end{align}  
 where $N = 2k$ and $\beta_0 = \beta_0 (k,N)$ as defined in \eqref{beta0}.
 \end{thm}

\begin{thm} \label{main pll 3}
 Let $\{ u_m : ||\Delta_g u_m||_{L^{\frac{N}{2}}(\hn)} = 1\}$ be a sequence in $W^{2, \frac{N}{2}}(\hn)$ such that 
 $u_m$ converges weakly to a nonzero function $u$ in $W^{2, \frac{N}{2}}(\hn).$ Then 
 \begin{align}
  \sup_m \int_{\hn} \Phi_{2, \frac{N}{2}}\left(\beta_0 p |u_m|^{\frac{N}{N - 2}}\right) \ dv_g < + \infty,
 \end{align}
 for all $p < \left(1 - ||\Delta_g u||^{\frac{N}{2}}_{\frac{N}{2}}\right)^{-\frac{2}{N -2}},$ 
 where $\beta_0 = \beta_0 (2,N)$ as defined in \eqref{beta0}, 
 \begin{align*}
 \Phi_{2 ,\frac{N}{2}}(t) := e^t - \displaystyle{\sum_{j = 0}^{j_{\frac{N}{2}} - 2} \frac{t^j}{j !}} \ \ 
 \ \mbox{and} \ \ j_{\frac{N}{2}} = \min \left\{j : j \geq \frac{N}{2}\right\}.
 \end{align*}
\end{thm}

The paper is organized as follows: In section 2 we introduced a few notations and definitions used in this article and 
recall some of the basic results. We prove one of our main lemma in section 3. Section 4 and 5 are devoted for the proofs 
of the theorems. In section 4 we will prove Theorems \ref{exact growth theorem 1}, \ref{exact growth theorem 2} 
and \ref{main pll 2}. Finally in section 5 we prove Theorem \ref{main pll 3} first and then Theorem \ref{main pll1}.

\section{Notations and Preliminaries}
 \textbf{Notations:}
 Let $\Omega$ be a bounded domain in $\rn.$ We denote by: \\
 $W^{k,p}(\Omega)$ : The usual Sobolev space with respect to the norm
 \begin{align*}
  ||u||_{W^{k,p}(\Omega)} := \left(\sum_{|\alpha| \leq k} ||D^{\alpha} u||_{L^p(\Omega)}\right)^{\frac{1}{p}},
 \end{align*}
where $\alpha = (\alpha_1, \alpha_2, ..., \alpha_N), \alpha_i \in \mathbb{N} \cup \{0\}$ is a multi-index,
$|\alpha| = \alpha_1 + \alpha_2 + ... + \alpha_N, D^{\alpha}$ is the weak derivative of order $\alpha,$ 
\begin{align*}
 D^{\alpha} := \frac{\partial^{|\alpha|}}{\partial x_1^{\alpha_1} \partial x_2^{\alpha_2}...\partial x_N^{\alpha_N}},
\end{align*}
and $||h||_{L^p(\Omega)}$ is the $L^p$-norm of $h$ in $\Omega.$ \\
$W^{k,p}_0(\Omega)$: the closure of $C^{\infty}_c(\Omega)$ in $W^{k,p}(\Omega).$ 

When $p = 2$ we will denote $W^{k,2}(\Omega)$ (respectively $W^{k,2}_0(\Omega)$) by $H^k(\Omega)$ (respectively $H^k_0(\Omega)$).

\vspace{0.5 cm}

\textbf{Hyperbolic space :}
The hyperbolic $N$-space is a complete, simply connected, Riemannian $N$-manifold having constant sectional
 curvature equals to $-1.$ It is well known that two manifolds having above properties are isometric (see \cite{Wolf}).
We will denote the hyperbolic $N$-space by $\hn.$ 

If we consider $\bn := \{x=(x_1,x_2,...,x_N) \in \rn : x^2_1 + x^2_2 + ... + x^2_N < 1 \}$ together with the Poincar\'{e}
metric $g$ given by
\begin{align}
 g:= \sum_{i = 1}^N \left(\frac{2}{1 - |x|^2}\right)^2 dx^2_i,
\end{align}
then it has constant sectional curvature equals to $-1.$ 
$(\bn, g)$	 is called the conformal ball model for the hyperbolic $N$-space.
There are several other models for the hyperbolic $N$-space,
for the rest of this article we will only work with the conformal ball model.

The volume element for the hyperbolic $N$-space is given by $dv_g = \left(\frac{2}{1 - |x|^2}\right)^N dx,$ where
$dx$ is the Lebesgue measure in $\rn.$
Let $\nabla, \Delta$ be the Euclidean gradient and Laplacian and $\langle,\rangle$ denotes the standard inner product 
in $\rn.$ Then in terms of local coordinate the hyperbolic gradient $\nabla_g$ and the Laplace-Beltrami operator $\Delta_g$ 
are given by :
\begin{align} \label{gradlaplacian}
 \nabla_g = \left(\frac{1 - |x|^2}{2}\right)^2\nabla \ , \ 
 \Delta_g = \left(\frac{1 - |x|^2}{2}\right)^2 \Delta + (N - 2)\left(\frac{1 - |x|^2}{2}\right) \langle x,\nabla \rangle,
\end{align}

\begin{defi}[Hyperbolic translation] 
Let $b \in \bn,$ then the hyperbolic translation $\tau_b : \bn \rightarrow \bn$ is given by,
\begin{align} \label{hyperbolictranslation}
  \tau_{b}(x) := \frac{(1 - |b|^2)x + (|x|^2 + 2\langle x,b \rangle + 1)b}{|b|^2|x|^2 + 2\langle x,b \rangle + 1}.
\end{align}
$\tau_b$ is an isometry from $\bn$ to $\bn.$ See \cite{Ratcliffe} for various other properties and isometries.
\end{defi}
We have the following useful lemma :
\begin{lem} \label{lemma1}
  Let $\tau_b$ be the hyperbolic translation of $\bn$ by $b.$ Then,
  \item[(i).] For all $u \in \Chn,$ there holds,
  \begin{align*}
   \Delta_g (u \circ \tau_b) = (\Delta_g u) \circ \tau_b, \ \ \langle \nabla_g (u \circ \tau_b), \nabla_g (u \circ \tau_b)\rangle_g = \langle(\nabla_g u) \circ \tau_b, (\nabla_g u) \circ \tau_b\rangle_g.
  \end{align*}
  \item[(ii).]For any $u\in \Chn$ and open subset $U$ of $\bn$
  \begin{align*}
  \int_{U} |u \circ \tau_b|^p \ dv_g = \int_{\tau_b(U)} |u|^p \ dv_g, \ \mbox{for all} \ 1 \leq p < \infty.
 \end{align*}
  
\end{lem}

\textbf{The Sobolev space $W^{k,p}(\hn)$ :} For a positive integer $l,$ let $\Delta^l_g$  denotes the $l$-th iterated 
Laplace-Beltrami operator. Define,
\begin{align} \label{iteratedgradient}
 \nabla^l_g :=
 \begin{cases}
  \Delta^{\frac{l}{2}}_g , \ \ \ \ \ \ \mbox{ if $l$ is even} \\
  \nabla_g \Delta^{\frac{l - 1}{2}}_g, \ \ \mbox{if $l$ is odd}.
 \end{cases}
\end{align}
Also define : 
\begin{align*}
  |\nabla^l_g u|_g :=
  \begin{cases}
   |\nabla^l_g u|, \ \ \ \ \ \ \ \ \ \ \ \mbox{if $l$ is even}, \\
   \langle \nabla^l_g u, \nabla^l_g u \rangle^{\frac{1}{2}}_g ,\ \ \mbox{if $l$ is odd}. 
  \end{cases}
\end{align*}

Then define the Sobolev space $W^{k,p}(\hn)$ as the closure of $\Chn$ functions with respect to the norm
\begin{align}
 ||u||_{W^{k,p}(\hn)} := \sum_{m = 0}^k \left[\int_{\hn} 
 |\nabla^m_g u|^p_g  \ dv_g\right]^{\frac{1}{p}},
\end{align}
We will denote $W^{k,2}(\hn)$ by $H^k(\hn).$ For the hyperbolic space it is known that the following Poincar\'{e}
type inequality holds :
\begin{align}
 \int_{\hn} |u|^p \ dv_g \leq C \int_{\hn} |\nabla^k_g u|^p_g  \ dv_g .
\end{align}
More generally if we define :
\begin{align}
 |||u|||_{W^{k,p}(\hn)} := \left[\int_{\hn} 
 |\nabla^k_g u|^p_g  \ dv_g\right]^{\frac{1}{p}},
\end{align}
then it turns out to be an equivalent norm on $W^{k,p}(\hn)$ (see \cite{Fontana15}, \cite{Tataru}).

\vspace{0.5 cm}

\textbf{GJMS operators on $\hn$ :}
Let $(M,g)$ be a Riemannian manifold of even dimension $N$ and $k$ be a positive integer less than or equals to 
$\frac{N}{2}.$

A GJMS operator $P_{k,g}$ of order $2k$ is a conformally invariant differential operators with leading term $\Delta^k_g$,
in the sense that : it satisfies
\begin{equation}\label{confrel}
P_{k,\tilde g}(v) = e^{-(\frac N2 +k)u}P_{k,g}(e^{(\frac N2 -k)u} v).
\end{equation}
for a conformal metric  $\tilde g = e^{2u}g.$

The existence of such higher order operators are due to Graham, Janne, Mason and Sparling (\cite{GJMS}) after 
a fourth order operator by Paneitz and subsequently a sixth order operator by Branson were discovered. 

For $k \in \{1,2,...,\frac{N}{2}\}$ the explicit expression for the GJMS operators are known for the hyperbolic 
space (see \cite{Liu}, \cite{juhl}), and it is given by :
\begin{align} \label{GJMS}
 P_{k,g} := P_1 (P_1 + 2) (P_1 + 6)...(P_1 + k(k - 1)),
\end{align}
where $P_1 := \left[- \Delta_g - \frac{N(N - 2)}{4} \right],$ is the Yamabe operator.

For simplicity of notations we will write GJMS operators on the hyperbolic space by $P_k.$ 
Therefore for all $u \in \Chn$ we can write from \eqref{GJMS} 
\begin{align} \label{expression for GJMS operator}
 \int_{\hn} (P_k u)u \ dv_g =  \int_{\hn} |\nabla^k_g u|^2_g \ dv_g + \sum^{k - 1}_{m = 0} (-1)^{k - m}  
 a_{km} \int_{\hn} |\nabla^m_g u|^2_g \ dv_g ,
\end{align}
where $a_{km}$ are non-negative constants.

The operators $P_k$ gives rise to a conformally invariant norm 
in the space $H^k(\hn)$ for $1\le k\le \frac N2.$ The next lemma will make the statement precise, whose 
proof can be found in \cite{SD}.
\begin{lem} \label{equivalence of norms}
Let $||u||_{k,g}$ be defined by
\begin{align} \label{norm}
 ||u||_{k,g} := \left[\int_{\hn} (P_k u)u \ dv_g\right]^{\frac 12}, \  \ u \in \Chn,
\end{align}
then $||.||_{k,g}$ defines a norm on $\Chn$. When $N=2k$ there exists a positive constant $\Theta$ such that,
\begin{align}\label{equivalent}
  \frac{1}{\Theta} ||u||_{k,g} \leq ||u||_{H^k(\hn)} \leq \Theta ||u||_{k,g},
 \end{align}
for all $u \in \Chn.$ 
\end{lem} 

\subsection{Basic lemmas :}
 We devote this subsection to recall some  basic lemmas already proved in \cite{SD}. We will also derive
 a few local estimates, which will be useful for the later analysis. \\
  For an open set $U\subset \bn$ and a positive integer $k$ define
\begin{align*}
 ||u||_{H^k_g(U)} :=  \left[\sum_{m = 0}^{k} \int_{U} |\nabla^m_g u|^2_g  \ dv_g\right]^{\frac 12}\;, u \in C^{k}(\overline{U}) . 
\end{align*}

We have the following lemmas from \cite{SD} : 
\begin{lem} \label{basicl2}
 Let $k$ be any positive integer, and $V,U$ be open sets such that $\overline{V} \subset U \subset \overline{U} \subset \bn$
 with $N = 2k$, 
 then there exists a constant $C_0 > 0$ such that 
 \begin{align}
  ||u||_{H^{k}(V)} \leq C_0 ||u||_{H^k_g(U)}, \ \ \mbox{for all} \ u \in C^{k}(\overline{U}).
 \end{align}
\end{lem}

\begin{lem} \label{basicl4}
 Let $U,V, N, k$ be as in Lemma \ref{basicl2}, then there exists $q > 0$ and a positive constant $C_2 > 0,$ such that for all
 $u \in C^{\infty}(\overline{U}),$ with $||u||_{H^k_g(U)} < 1,$ satisfies,
 \begin{align}
  \int_V \left( e^{qu^2} - 1 \right ) \ dx \leq C_2
  \frac{||u||^2_{H^k_g(U)}}{1 - ||u||^2_{H^k_g(U)}}.
 \end{align}
\end{lem}

We now require few more local estimates in the spirit of Lemma \ref{basicl4}. For the next two lemmas
we will restrict ourselves to dimension $4$ only.

\begin{lem} \label{Adams inequality with exact growth lemma 2}
 Let $V,U $ be open sets in $\mathbb{B}^4$ with smooth boundary such that 
 $\overline{V} \subset U \subset \overline{U} \subset \mathbb{B}^4,$  
 then there exists a constant $q_1 > 0$ such that for all  $u \in C^{\infty}(\overline{U})$ 
 with $||u||_{H^2_g(U)} \leq 1,$ there holds,
 \begin{align}
  \int_V \frac{e^{q_1 u^2} - 1}{(1 + \alpha_0 |u|)^2} \ dx \leq C_1 \int_{U} u^2 \ dv_g,
 \end{align}
where $C_1 $ is a positive constant and $\alpha_0 = \sqrt{\frac{q_1}{32 \pi^2}}.$
 \end{lem}

 \begin{proof}
 Using lemma \ref{basicl2} we get, for the above choice of $V$ and $U,$ there exists a constant $C_0 > 0$ such that
 \begin{align*}
  ||u||_{H^{2}(V)} \leq C_0 ||u||_{H^2_g(U)}, \ \ \mbox{for all} \ u \in C^{\infty}(\overline{U}).
 \end{align*}

  There exists an  extension operator $T$ from $H^2(V)$ to $H^2_0(\mathbb{B}^4)$ such that for all 
  $u \in C^{\infty}(\overline{U})$ we have,
  \begin{equation} \label{required estimates1}
   \begin{cases}
   T(u) & \equiv u \ \ \mbox{on} \ V, \\
  ||T(u)||_{L^2(\mathbb{B}^4)} &\leq C_2 ||u||_{L^2(V)}, \ \ \mbox{for some constant} \ C_2 >0, \\
  ||\Delta T(u)||_{L^2(\mathbb{B}^4)} &\leq ||T||||u||_{W^{2,2}(V)} \leq C_0||T|| \ ||u||_{H^2_g(U)}
  \end{cases}
  \end{equation}
Let us fix $\frac{1}{\alpha_0} = C_0||T||,$ then for all $u \in C^{\infty}(\overline{U})$ with $||u||_{H^2_g(U)} \leq 1,$
we have
\begin{align}
 \int_{\bn} |\Delta \left(\alpha_0 T(u)\right)|^2 \ dx \leq 1.
\end{align}
Then by theorem \eqref{MasoudiSani1} we have,
\begin{align} \label{required estimates2}
 \int_{\mathbb{B}^4} \frac{e^{32\pi^2 \alpha_0^2 T(u)^2} - 1}{(1 + \alpha_0|T(u)|)^2} \ dx \leq C ||T(u)||_{L^2(\mathbb{B}^4)}.
\end{align}
Therefore setting $q_1 = 32 \pi^2 \alpha_0^2,$ we have from \eqref{required estimates1} and
\eqref{required estimates2},
   
\begin{align*}
 \int_V \frac{e^{q_1 u^2} - 1}{(1 + \alpha_0 |u|)^2} \ dx &\leq C_1 ||u||^2_{L^2(V)} \\
                                                       &\leq C_1 \int_{U} u^2 \ dv_g .
\end{align*}
This completes the proof.
 \end{proof}

\begin{lem} \label{Adams inequality with exact growth lemma 3}
Let $V,U $ be open sets in $\mathbb{B}^4$ with smooth boundary such that 
 $\overline{V} \subset U \subset \overline{U} \subset \mathbb{B}^4,$  
 then there exists a constant $q_2 > 0$ such that for all  $u \in C^{\infty}(\overline{U})$ 
 with $||u||_{H^2_g(U)} \leq 1,$ there holds,
 \begin{align}
  \int_V \left(e^{q_2 u^2} - 1 \right) \ dx \leq \tilde C_1 \int_{U} u^2 \ dv_g .
 \end{align}
\end{lem}

The proof of Lemma \ref{Adams inequality with exact growth lemma 3} goes in the same line as in 
Lemma \ref{Adams inequality with exact growth lemma 2}, so we can omit the proof. 

Finally we end this section with the following covering lemma which will help, together with the above local estimates,
to bound the integral near infinity (see \cite{SD}, \cite{AdiT} for a proof). 
\begin{lem} \label{basicl5}
 Let $U,V$ be any open sets in $\bn$ such that, $\overline{V} \subset U \subset \overline{U} \subset \bn,$ then there
 exists a countable collection $\{b_i\}^{\infty}_{i = 1}$ of elements in $\bn,$ and  a positive number $M_0 \in \mathbb{N},$
 such that, 
\item[(i).] $\{\tau_{b_i}(V)\}^{\infty}_{i = 1}$ covers $\bn$ with multiplicity not exceeding $M_0,$
\item[(ii).] $\{\tau_{b_i}(U)\}^{\infty}_{i = 1}$  have multiplicity not exceeding $M_0.$
\end{lem}

\section{Main Lemma}
In this section we will prove two basic lemmas which will be needed for the proof of theorem \ref{main pll1} and 
theorem \ref{main pll 3}. Before we proceed let us recall a few things about decreasing rearrangement
with respect to the Lebesgue measure in the Euclidean space. 

\vspace{0.5 cm}

Let $\Omega$ be a bounded open set in $\rn,$ and let $f$ be a real valued measurable function defined on $\Omega.$
For a measurable set $A$ in $\rn,$ let $|A|$ denotes the Lebesgue measure of the set $A.$ Let $\mu_f$ denotes
the distribution function of $f,$
\begin{align*}
 \mu_f(t) := |\{x \in \Omega : |f(x)| > t\}|.
\end{align*}
The decreasing rearrangement $f^{\#}$ of $f$ is defined as:
\begin{align*}
 f^{\#}(s) := \sup \{t \geq 0 : \mu_f(t) > s \}, \ \ s \in [0, |\Omega|].
\end{align*}
Let $\Omega^{*}$ be the ball with centre at the origin and having same measure as $\Omega.$ Then the spherically
symmetric decreasing rearrangement $f^{*}$ of $f$ is defined by,
\begin{align*}
 f^{*}(x) := f^{\#}(\sigma_N |x|^N), \ \ x \in \Omega^{*}.
\end{align*}
where $\sigma_N$ is the volume of the unit ball in $\rn.$

One can also define spherically symmetric decreasing rearrangement for a measurable function $f$ defined on $\rn$ which 
vanishes at infinity. 

Let $f : \rn \rightarrow \mathbb{R}$ be a measurable function, $f$ is said to vanishes at infinity if 
$|\{ |f| > t\}|$ has finite measure for all $t > 0.$ Then define,
\begin{align*}
 f^{*}(x) := \int_0^{\infty} \chi_{\{|f| > t\}^{*}} (x) \ dt,
\end{align*}

For a profound discussion on the symmetric decreasing rearrangement we refer to \cite{Lieb}, \cite{Kesavan}.

Now we will be able to state and prove a comparison result which will enable us to estimate 
the integral in the interior. This is the main ingredient to prove Theorem \ref{main pll 3}.

\begin{lem} \label{main lemma}
 Let $0 < R < 1$ and $V = B(0,R),$ then there exists a constant $c_0 >0,$ such that for $u \in \Cbn,$ there exists a
 function $v \in C^{2,\gamma}_{loc}(\bn),$ satisfying the following properties :
 \item[(i).] $\int_{\hn} |\Delta_g u|^{\frac{N}{2}} \ dv_g = \int_{\bn} |\Delta v|^{\frac{N}{2}} \ dx,$
 \item[(ii).] $ v \equiv 0$ on $\partial \bn$
 \item[(iii).] For any $\beta > 0,$ if we set $c(u) := c_0 ||\Delta_g u||_{L^{\frac{N}{2}}(\hn)},$ then 
 \begin{align}
  \int_V \left(e^{\beta |u|^{\frac{N}{N - 2}}} - 1\right) \ dx \leq
  \int_V \left(e^{\beta \left |v + c(u) \right |^{\frac{N}{N - 2}}} - 1\right) \ dx
 \end{align}
\end{lem}
\begin{proof}
 Let $u \in \Cbn,$ then for all $x \in \bn,$
 \begin{align}
  \Delta_g u (x) = \left(\frac{1 - |x|^2}{2}\right)^2 \Delta u + 
  (N - 2) \left(\frac{1 - |x|^2}{2}\right) \langle x, \nabla u (x)\rangle 
   = \left(\frac{1 - |x|^2}{2}\right)^2 f(x),
 \end{align}
where,
\begin{align} \label{definef}
 f(x) = \Delta u(x) + (N - 2)\left(\frac{2}{1 - |x|^2}\right) \langle x, \nabla u(x)\rangle .
\end{align}
Since $u$ has compact support in $\bn,$ we have $f \in \Cbn,$ and 
\begin{align}
 \int_{\bn} |f|^{\frac{N}{2}} \ dx = \int_{\hn} |\Delta_g u|^{\frac{N}{2}} \ dv_g.
\end{align}
Let $v$ be the solution of the problem :
\begin{align} \label{definev}
  - \Delta v  &= |f|^{*} \ \ \ \mbox{in} \ \bn  \notag \\
   v &= 0 \ \ \ \ \ \ \mbox{on} \ \partial \bn.
\end{align}
  
Since $f \in \Cbn,|f|$ is atleast a Lipschitz function and hence $|f|^{*}$ is also Lipschitz. Therefore by 
elliptic regularity we conclude that $v \in C^{2,\gamma}_{loc}(\bn),$ for some $\gamma > 0.$ Then one can write $v$ as,
\begin{align}
 v(x) = \int_{\bn} G(x,y) |f|^{*}(y) \ dy ,
\end{align}
where $G$ is the Green's function for the Laplacian in the ball $\bn$ with Dirichlet boundary condition.
We claim that this $v$ is the desired function. To prove this let $U$ be an open set in $\bn$ such that
$\overline{V} \subset U \subset \overline{U} \subset \bn.$ Let $\psi$ be a smooth cut-off function such that,
\begin{align*}
 0 \leq \psi \leq 1, \ \psi \equiv 1 \ \mbox{on} \ \overline{V}, \ \mbox{and} \ \mbox{supp} \ \psi \subset U.
\end{align*}
Therefore we can write $\psi u$ as
\begin{align} \label{psiu}
 (\psi u )(x) = \int_{\bn} F(x - y) (- \Delta) (\psi u)(y) \ dy,
\end{align}
where $F$ is the fundamental solution of the Laplacian on $\rn.$
Now we can expand $ \Delta (\psi u)(y)$ and it gives,
\begin{align} \label{deltak}
 - \Delta (\psi u)(y) = \psi(y)  \Delta (u)(y) + 2\langle \nabla \psi (y), \nabla u (y)\rangle + u(y) \Delta \psi (y) ,
\end{align}
Since supp$(\psi)$ is contained in $U,$ we get from \eqref{psiu} and \eqref{deltak},
\begin{align} \label{symmetry1}
 |(\psi u)(x)| &\leq \int_{\bn} F(x - y) \psi(y)|\Delta u(y)| \ dy  
 +  \int_U F(x - y)|2\langle \nabla \psi, \nabla u\rangle + u \Delta \psi| \ dy \notag \\
 &\leq \int_{\bn} F(x - y) \psi(y)|\Delta u(y)| \ dy 
 + C \sum_{|\alpha| \leq 1} \int_U |x - y|^{2 - N} |D^{\alpha} u(y)| \ dy.
\end{align}
Now from \eqref{definef} we get,
\begin{align*}
|\Delta u(y)| \leq |f(y)| + (N - 2) \left(\frac{2}{1 - |y|^2}\right) |\nabla  u(y)|,
\end{align*}
which together with \eqref{symmetry1} gives,
\begin{align} \label{symmetry2}
 |(\psi u)(x)| &\leq \int_{\bn} F(x - y)|f(y)| \ dy 
 + C\sum_{|\alpha| \leq 1} \int_U |x - y|^{2 - N} |D^{\alpha} u(y)| \ dy.
\end{align}
We will first estimate the second term in \eqref{symmetry2}.  
Let $\alpha$ be any multi index such that $|\alpha| \leq 1,$ and $\phi$ be any function such that 
$\phi \in W^{2, \frac{N}{2}}(U),$ then $D^{\alpha} \phi \in W^{1, \frac{N}{2}}(U).$ Therefore by
Sobolev embedding theorem we have,
\begin{align} \label{GNS}
 \int_U |D^{\alpha} \phi|^N \ dy &\leq C || D^{\alpha} \phi ||^N_{W^{1,\frac{N}{2}}(U)} 
 \leq C ||\phi ||^N_{W^{2,\frac{N}{2}}(U)}.
\end{align}
Therefore from \eqref{symmetry2} we get,
\begin{align} \label{symmetry3}
 |(\psi u)(x)| &\leq \int_{\bn} F(x - y)|f(y)| \ dy \notag \\
 &+ C\sum_{|\alpha| \leq 1} \left(\int_U |x - y|^{(2 - N)\frac{N}{N - 1}} \ dy \right)^{\frac{N - 1}{N}}
 \left(\int_U |D^{\alpha} u|^N \ dy \right)^{\frac{1}{N}}
\end{align}
Now we see that $U \subset B(x,2)$ when $|x| < 1,$ so that,
\begin{align*}
 \int_U |x - y|^{(2 - N)\frac{N}{N - 1}} \ dy &\leq \int_{B(x,2)} |x - y|^{(2 - N)\frac{N}{N - 1}} \ dy \\
 &\leq C \int_{\{|y| < 2\}} |y|^{(2 - N)\frac{N}{N - 1}} \ dy \leq C.
\end{align*}
From \eqref{symmetry3} and \eqref{GNS} and using $\psi \equiv 1$ on $V,$ we get,
\begin{align} \label{symmetry4}
 | u(x)| &\leq \int_{\bn} F(x - y)|f(y)| + C ||u||_{W^{2,\frac{N}{2}}(U)}, \ \ \mbox{for all} \ x \in V
\end{align}

Let $r$ be any real number such that $0 < r \leq R,$ and $\phi \geq 0$ be a smooth function such that supp($\phi$) 
lies in $B(0,r).$ Multiplying \eqref{symmetry4} by $\phi$ and integrating over $B(0,r)$ we get,
\begin{align}\label{symmetry5}
 \int_{B(0,r)} |u(x)|\phi(x) \ dx &\leq \int_{B(0,r)} \int_{\bn} \phi(x) F(x - y)|f(y)| \ dy dx \notag \\
 &+ C ||u||_{W^{2,\frac{N}{2}}(U)} \int_{B(0,r)} \phi(x) \ dx,
\end{align} 
Since supp($f$) in contained in $\bn,$ extending by zero outside $\bn$ we can write \eqref{symmetry5} as,
\begin{align} \label{symmetry6}
 \int_{B(0,r)} |u(x)| \phi(x) \ dx &\leq \int_{\rn} \int_{\rn} \phi(x) F(x - y)|f(y)| \ dy dx \notag \\
 &+ C ||u||_{W^{2,\frac{N}{2}}(U)} \int_{B(0,r)} \phi(x) \ dx,
\end{align}

Now applying Riesz inequality (\cite{Lieb}, Theorem 3.7) for the convolution of symmetric decreasing rearrangement 
we get from \eqref{symmetry6},
\begin{align} \label{symmetry7}
 \int_{B(0,r)} |u(x)| \phi(x) \ dx &\leq \int_{\rn} \int_{\rn} \phi^{*}(x) F^{*}(x - y)|f|^{*}(y) \ dy dx \notag \\
 &+ C ||u||_{W^{2,\frac{N}{2}}(U)} \int_{B(0,r)} \phi(x) \ dx, \notag \\
 &\leq \int_{\rn} \int_{\rn} \phi^{*}(x) F(x - y)|f|^{*}(y) \ dy dx \notag \\
 &+ C ||u||_{W^{2,\frac{N}{2}}(U)} \int_{B(0,r)} \phi(x) \ dx, \notag \\
 &\leq \int_{B(0,r)} \phi^{*}(x) \int_{\bn} F(x - y)|f|^{*}(y) \ dy dx \notag \\
 &+ C ||u||_{W^{2,\frac{N}{2}}(U)} \int_{B(0,r)} \phi(x) \ dx,
\end{align}

Now for $0 < |x| \leq R$ and $y \in \bn,$ we have $|G(x,y) - F(x - y)| \leq C_1(R),$ and so,
\begin{align} \label{mainsymmetry}
 \int_{\bn} F(x - y)|f|^{*}(y) \ dy &\leq \int_{\bn} G(x,y)|f|^{*}(y) \ dy 
 + C_1(R) \int_{\bn} |f|^{*}(y) \ dy \notag \\
 &\leq v(x) + C_2 \left(\int_{\bn} (|f|^{*}(y))^{\frac{N}{2}} \ dy\right)^{\frac{2}{N}} \notag \\
 &\leq v(x) + C_2 \left(\int_{\hn} |\Delta_g u|^{\frac{N}{2}} \ dv_g \right)^{\frac{2}{N}}.
\end{align}
One can easily show that there exists a constant $C>0$ such that for all $u \in \Cbn,$ there holds
\begin{align} \label{mainsymmetry1}
 ||u||_{W^{2,\frac{N}{2}}(U)} \leq C ||\Delta_g u||_{L^{\frac{N}{2}}(\hn)}
\end{align}

Hence from \eqref{symmetry7},\eqref{mainsymmetry} and \eqref{mainsymmetry1} we get,
\begin{align} \label{symmetry11}
 \int_{B(0,r)} |u(x)| \phi(x) \ dx \leq \int_{B(0,r)} \phi^{*}(x) v(x) \ dx + 
 c_0||\Delta_g u||_{L^{\frac{N}{2}}(\hn)}\int_{B(0,r)} \phi^{*}(x) \ dx.
\end{align}
Setting $c(u) = c_0||\Delta_g u||_{L^{\frac{N}{2}}(\hn)},$ we get from \eqref{symmetry11},
\begin{align} \label{symmetry12}
 \int_{B(0,r)} |u(x)| \phi(x) \ dx \leq \int_{B(0,r)} \phi^{*}(x) (v(x) + c(u)) \ dx, \ \ \mbox{for all} \ 0 < r \leq R.
\end{align}

Since the above inequality \eqref{symmetry12} is true for all nonnegative smooth functions $\phi$ having support
inside $B(0,R)$ and $L^p$-norm remains unchanged under symmetric decreasing rearrangement we have,
\begin{align} \label{finalsymmetry}
 \int_{V} |u|^{p} \ dx \leq \int_{V} |v + c(u)|^p \ dx, \ \ \ \mbox{for all} \ 1 \leq p < + \infty.
\end{align}
Hence we have from \eqref{finalsymmetry}
\begin{align} \label{finalcomparison}
  \int_V  \left(e^{\beta |u|^{\frac{N}{N - 2}}} - 1\right) \ dx \leq 
  \int_V  \left(e^{\beta |v + c(u)|^{\frac{N}{N - 2}}} - 1\right) \ dx,
\end{align}
where we used both the integral in \eqref{finalcomparison} is finite. This complete the proof of 
the lemma.
\end{proof}

We also require the following lemma to estimate the integral at infinity. 
\begin{lem} \label{test function}
  Given a positive integer $k,$ there exists a constant $C_0 > 0,$ such that for any $R>0$ and $\delta > 0,$ 
  there exists $\phi_{R,\delta}$ such that,
  \begin{align*}
  (i). \ \ &\phi_{R,\delta} \in C^{\infty}_c(B(0, R + \delta)),  \qquad \qquad \qquad \qquad \qquad \qquad \qquad \qquad
  \qquad \qquad \qquad \\ 
  (ii). \ \ &0 \leq \phi_{R,\delta} \leq 1,\ \mbox{and} \ \phi_{R,\delta} \equiv 1 \ \mbox{on} \ B(0,R), \\
  (iii). \ \ &\sum_{1 \leq |\alpha| \leq 2k} \delta^{|\alpha|} ||D^{\alpha} \phi_{R,\delta}||_{L^{\infty}} \leq C_0.  
  \end{align*}
 \end{lem}
\begin{proof}
 Let $\eta \in C^{\infty}_c(\bn)$ be such that $\eta \geq 0, \ $  $\int_{\bn} \eta(x) \ dx = 1,$ and define 
 \begin{align*}
  \eta_{\epsilon} (x) = \frac{1}{\epsilon^N} \eta \left(\frac{x}{\epsilon}\right), \ \ \mbox{where} \ 
  \epsilon = \frac{\delta}{10}.
 \end{align*}
For the given $R$ and $\delta$ define,
\begin{align*}
 f(t) =
 \begin{cases}
  1, \qquad \qquad \ \ \ \  \ \ \mbox{if} \ 0 \leq t \leq R + \frac{\delta}{3} ,\\
  -\frac{3}{\delta}t + \frac{3R + 2 \delta}{3},\ \ \ \mbox{if} \ R + \frac{\delta}{3} < t \leq R + \frac{2\delta}{3}, \\
  0, \qquad \qquad \ \ \ \ \ \ \ \mbox{if} \ t \geq R + \frac{2\delta}{3} ,
  \end{cases}
\end{align*}
and define $\phi(x) = f(|x|).$ Then one has $\phi$ is Lipschitz, $0 \leq \phi \leq 1, \phi \equiv 1$ on 
$B(0, R + \frac{\delta}{3})$ and $supp(\phi)$ is contained in $B(0, R + \frac{2\delta}{3}).$ Finally define
\begin{align}
 \phi_{R, \delta} (x) = \phi * \eta_{\epsilon} (x).
\end{align}
Then it is easy to see that $\phi_{R, \delta}$ is of class $C^{\infty}, 0 \leq \phi_{R, \delta} \leq 1,$ 
and $supp(\phi_{R, \delta})$ is contained in $B(0, R + \delta).$ 

Now for $|x| < R,$ and $|y| < \epsilon,$ one has $|x - y| < R + \frac{\delta}{3},$ and therefore,
\begin{align*}
 \phi_{R, \delta}(x) &= \int_{B(0, \epsilon)} \phi(x - y) \eta_{\epsilon}(y) \ dy \\
                     &= \int_{B(0, \epsilon)} \eta_{\epsilon} (y) \ dy = 1.
\end{align*}
This proves $(i)$ and $(ii)$ of the lemma. For the last part one can easily check that for any multi-index $\alpha,$
\begin{align*}
 D^{\alpha} \phi_{R, \delta}(x) = \frac{1}{\epsilon^{|\alpha|}} \int_{\bn} \phi(x - \epsilon y)D^{\alpha} \eta(y) \ dy,
\end{align*}
and therefore we have,
\begin{align*}
 \delta^{|\alpha|}||D^{\alpha} \phi_{R, \delta}||_{L^{\infty}}
 \leq 10^{|\alpha|} ||D^{\alpha} \eta||_{L^1(\bn)}.
\end{align*}
Taking $C_0 = \sum_{1 \leq |\alpha| \leq 2k} 10^{|\alpha|} ||D^{\alpha} \eta||_{L^1(\bn)},$ we have the desired 
estimate, and this completes the proof of the lemma.
\end{proof}
 
 \section{Proof of Theorems : Part I}
 \textbf{Proof of Theorem \ref{exact growth theorem 1} :}
 \begin{proof}
 Let $V,U,q_1, \alpha_0$ be as in Lemma \ref{Adams inequality with exact growth lemma 2}. By covering Lemma \ref{basicl5}, 
 we can find a countable collection $\{b_i\}_{i = 1}^{\infty}$ such that $\{\tau_{b_i}(V)\}$ covers $\mathbb{B}^4$,
 and $\{\tau_{b_i}(U)\}$ have finite multiplicity, say $M_0.$ 
  
 Now fix $u \in C^{\infty}_c(\mathbb{H}^4)$ such that $\int_{\mathbb{H}^4} P_2(u)u \ dv_g \leq 1,$ and define,
 \begin{align*}
  \mathcal{B}_u := \{ b_i : ||(u \circ \tau_{b_i})||_{H^2_g(U)} > \alpha_0 \}.
 \end{align*}
Then one can easily check that,
\begin{align} \label{cardinality of Au}
  Card( \mathcal{B}_u) \leq \frac{M_0}{\alpha^2_0} \left[||\Delta_g u||^2_{L^2(\mathbb{H}^4)}
 + ||\nabla_g u||^2_{L^2(\mathbb{H}^4)} + ||u||^2_{L^2(\mathbb{H}^4)}\right].
\end{align}

Now by Lemma \ref{equivalence of norms} there exists a constant $\Theta > 0,$ such that,
\begin{align*}
 ||\Delta_g u||^2_{L^2(\mathbb{H}^4)}
 + ||\nabla_g u||^2_{L^2(\mathbb{H}^4)} + ||u||^2_{L^2(\mathbb{H}^4)} \leq \Theta^2 \int_{\mathbb{H}^4} P_2(u)u \ dv_g 
 \leq \Theta^2.
\end{align*}

Therefore from \eqref{cardinality of Au} we have,
\begin{align*}
 Card( \mathcal{B}_u) \leq \frac{M_0}{\alpha^2_0} \Theta^2 .
\end{align*}
Now let $b_i$ be an element of the collection such that $b_i \notin \mathcal{B}_u.$ Then we have,
\begin{align*}
 ||\frac{1}{\alpha_0}(u \circ \tau_{b_i})||_U \leq 1.
\end{align*}
Therefore by Lemma \ref{Adams inequality with exact growth lemma 2} we have,
\begin{align}
 \int_{V} \frac{e^{\frac{q_1}{\alpha_0^2}(u \circ \tau_{b_i})^2} - 1}{(1 + |u \circ \tau_{b_i}|)^2} \ dx
 &\leq C_1 \int_{U} |u \circ \tau_{b_i}|^2 \ dv_g , \notag \\
 &= C_1 \int_{\tau_{b_i}(U)} u^2 \ dv_g .
\end{align}
Since $\frac{q_1}{\alpha_0^2} = 32 \pi^2,$ we have,
\begin{align*}
 \sum_{b_i \notin \mathcal{B}_u} \int_{\tau_{b_i}(V)} \frac{e^{32 \pi^2 u^2} - 1}{(1 + |u|)^2} \ dv_g 
 &= \sum_{b_i \notin \mathcal{B}_u} \int_{V} \frac{e^{32 \pi^2 (u \circ \tau_{b_i})^2} - 1}{(1 + |u \circ \tau_{b_i}|)^2} \ dv_g \\
 &\leq C \sum_{b_i \notin \mathcal{B}_u} \int_{V} \frac{e^{32 \pi^2 (u \circ \tau_{b_i})^2} - 1}{(1 + |u \circ \tau_{b_i}|)^2} \ dx \\
 &\leq C C_1 \sum_{b_i \notin \mathcal{B}_u} \int_{\tau_{b_i}(U)} u^2 \ dv_g \\
 &\leq \tilde C M_0 ||u||^2_{L^2(\mathbb{H}^4)}.
\end{align*}
Now if $b_i \in \mathcal{B}_u$ then,
\begin{align*}
 \int_{\tau_{b_i}(V)} \frac{e^{32 \pi^2 u^2} - 1}{(1 + |u|)^2} \ dv_g
 &= \int_{V} \frac{e^{32 \pi^2 (u \circ \tau_{b_i})^2} - 1}{(1 + |u \circ \tau_{b_i}|)^2} \ dv_g \\
 &\leq C_2 \int_{V} \frac{e^{32 \pi^2 (u \circ \tau_{b_i})^2} - 1}{(1 + |u \circ \tau_{b_i}|)^2} \ dx \\
 &\leq C_3 ||u \circ \tau_{b_i}||^2_{L^2(\mathbb{B}^4)} \\
 &\leq C_3 ||u||^2_{L^2(\mathbb{H}^4)},
\end{align*}
where we used the fact that $\int_{\mathbb{H}^4} P_2(u)u \ dv_g = \int_{\mathbb{B}^4} |\Delta u|^2 \ dx .$
Therefore adding such finitely many $b_i$'s we get,
\begin{align*}
 \sum_{b_i \in \mathcal{B}_u} \int_{\tau_{b_i}(V)} \frac{e^{32 \pi^2 u^2} - 1}{(1 + |u|)^2} \ dv_g 
 \leq C_4 ||u||^2_{L^2(\mathbb{H}^4)} .
\end{align*}
Since $\{\tau_{b_i}(V)\}^{\infty}_{i = 1}$ covers $\mathbb{B}^4$ we have,
\begin{align*}
 &\int_{\mathbb{H}^4} \frac{e^{32 \pi^2 u^2} - 1}{(1 + |u|)^2} \ dv_g  \\ 
 & \leq \sum_{b_i \in \mathcal{B}_u} \int_{\tau_{b_i(V)}} \frac{e^{32 \pi^2 u^2} - 1}{(1 + |u|)^2} \ dv_g
 +  \sum_{b_i \notin \mathcal{B}_u} \int_{\tau_{b_i(V)}} \frac{e^{32 \pi^2 u^2} - 1}{(1 + |u|)^2} \ dv_g \\
 &\leq C ||u||^2_{L^2(\mathbb{H}^4)}.
\end{align*}
This completes the proof.
\end{proof}

\textbf{Proof of Theorem \eqref{exact growth theorem 2} :}
\begin{proof}
 Let $V,U, q_2$ be as in Lemma \ref{Adams inequality with exact growth lemma 3} and $\{b_i\}_{i = 1}^{\infty}, M_0$
 be as in the proof of Theorem \ref{exact growth theorem 1}.
 
 Fix $\alpha \in (0,32 \pi^2), u \in C^{\infty}_c(\mathbb{H}^4)$ with $\int_{\mathbb{H}^4} P_2(u)u \ dv_g \leq 1,$ 
 and define,
 \begin{align*}
  \mathcal{C}_u := \left\{b_i : ||u \circ \tau_{b_i}||^2_{H^2_g(U)} > \frac{q_2}{\alpha}\right\}.
 \end{align*}
 Then as before one can check that,
\begin{align*}
 Card(\mathcal{C}_u) \leq \frac{\alpha M_0 }{q_2} \Theta^2.
\end{align*}
Therefore  using Lemma \ref{Adams inequality with exact growth lemma 3}, we have for all $b_i \notin \mathcal{C}_u,$
\begin{align*}
 \int_{\tau_{b_i}(V)} \left(e^{\alpha u^2} - 1 \right) \ dv_g
 &= \int_{V} \left(e^{\alpha (u \circ \tau_{b_i})^2} - 1 \right) \ dv_g, \\
 &\leq C \int_{V} \left(e^{q_2 \left(\sqrt{\frac{\alpha}{q_2}}u \circ \tau_{b_i}\right)^2} - 1 \right) \ dx, \\
 &\leq C \int_{U} |u \circ \tau_{b_i}|^2 \ dv_g , \\
 &\leq C \int_{\tau_{b_i}(U)} u^2 \ dv_g .
\end{align*}
Adding all such $b_i$'s we get,
\begin{align*}
 \sum_{b_i \notin \mathcal{C}_u} \int_{\tau_{b_i}(V)} \left(e^{\alpha u^2} - 1 \right) \ dv_g
 \leq CM_0 ||u||^2_{L^2(\mathbb{H}^4)}.
\end{align*}
If $b_i \in \mathcal{C}_u,$ then using $||\Delta (u \circ \tau_{b_i})||_{L^2(\mathbb{B}^4)} \leq 1,$ we have,
\begin{align*}
 \sum_{b_i \in \mathcal{C}_u} \int_{\tau_{b_i}(V)} \left(e^{\alpha u^2} - 1 \right) \ dv_g
 &\leq C \sum_{b_i \in \mathcal{C}_u} \int_{V} \left(e^{\alpha (u \circ \tau_{b_i})^2} - 1 \right) \ dx, \\
 &\leq C\sum_{b_i \in \mathcal{C}_u} ||u \circ \tau_{b_i}||^2_{L^2(\mathbb{B}^4)}, \\
 &\leq C\Big[ Card(\mathcal{C}_u) + 1 \Big] ||u||^2_{L^2(\mathbb{H}^4)}.
\end{align*}
This completes the proof.
\end{proof}

\textbf{Proof of Sharpness of Theorem \ref{exact growth theorem 1} and Theorem \ref{exact growth theorem 2} :}
\begin{proof}
 In order to complete the proof of sharpness of Theorem \ref{exact growth theorem 1} and 
 Theorem \ref{exact growth theorem 2}, it is enough to show that,
 \begin{align} \label{sharpness}
  \int_{\mathbb{H}^4} \frac{(e^{32 \pi^2 u^2} - 1)}{(1 + |u|)^p} \ dv_g \leq C ||u||^2_{L^2(\mathbb{H}^4)}
 \end{align}
does not hold for any $p < 2.$

Suppose if possible \eqref{sharpness} holds. Let us consider the following sequence of functions,
\begin{align}
 v_m(x):=
 \begin{cases}
  \sqrt{\frac{\log m}{16 \pi^2}} + \frac{1}{\sqrt{16 \pi^2 \log m}}(1 - m|x|^2), \ &\mbox{if} \ 0 \leq |x| \ \leq \frac{1}{\sqrt{m}},\\
  \frac{1}{\sqrt{4 \pi^2 \log m}} \log \frac{1}{|x|},  &\mbox{if} \ \frac{1}{\sqrt{m}} < |x| \leq 1, \\
  \xi_m(x),   &\mbox{if} \ |x| \geq 1,
 \end{cases}
\end{align}
where $\xi_m$'s are smooth functions satisfying,
\begin{align*}
\xi_m |_{\partial B_1(0)} = \ & 0 = \xi_m |_{\partial B_2(0)} \\
\frac{\partial \xi_m}{\partial r} |_{\partial B_1(0)} = \frac{1}{\sqrt{4 \pi^2 \log m}} &,
\frac{\partial \xi_m}{\partial r} |_{\partial B_2(0)} = 0,
\end{align*}
and $\xi_m, \Delta \xi_m$ are all $O\left(\frac{1}{\sqrt{\log m}}\right).$

Define $\tilde u_m(x) = v_m (3x),$ for $|x| \leq \frac{2}{3}$ and extend  it by zero outside
$\{|x| \leq \frac{2}{3}\}$. We can easily show that,
\begin{align*}
 ||\tilde u_m||^2_{L^2(\mathbb{H}^4)} = O\left(\frac{1}{\log m}\right),
 \int_{\mathbb{H}^4} P_2(\tilde u) \tilde u \ dv_g = 1 +  O\left(\frac{1}{\log m}\right).
\end{align*}
 Proceeding as in \cite{MSani} we can conclude that, if \eqref{sharpness} holds then,
\begin{align*}
 \limsup_{m \rightarrow \infty} \ \log m \int_{\mathbb{H}^4} \frac{(e^{32 \pi^2 u^2_m} - 1)}{(1 + |u_m|)^p} \ dv_g < +\infty,
\end{align*}
where, $u_m(x) := \frac{\tilde u_m}{\left[\int_{\mathbb{H}^4} P_2(\tilde u_m) \tilde u_m \ dv_g \right]^{\frac{1}{2}}}.$

Whereas, a direct computation gives,
\begin{align*}
 \log m \int_{\mathbb{H}^4} \frac{(e^{32 \pi^2 u^2_m} - 1)}{(1 + |u_m|)^p} \ dv_g
 \geq C \left(\log m \right)^{1 - \frac{p}{2}} \rightarrow + \infty,
\end{align*}
as $m \rightarrow + \infty.$ This completes the proof of sharpness.
\end{proof}
Hence the proof of Theorem \ref{exact growth theorem 1} and Theorem \ref{exact growth theorem 2} is completed.
\pagebreak

\textbf{Proof of Theorem \ref{main pll 2} :}

\begin{proof}
 Fix $p$ satisfying the condition of the theorem. Since $u_m$ converges weakly to $u$ in $H^k(\hn),$ there exists 
 $M > 0$ sufficiently large such that,
 \begin{align}
  ||u_m - u||^2_{k,g} < \frac{1}{p}, \ \ \ \ \ \ \mbox{for all} \ m \geq M.
 \end{align}
 Let $V,U, q$ be as in Lemma \ref{basicl4}. Then by covering Lemma \ref{basicl5}, there exists a collection 
 $\{b_i\}_{i = 1}^{\infty} \subset \bn $ and a positive number $M_0$ such that $\{\tau_{b_i}(V)\}_{i=1}^{\infty}$
 covers $\bn$ and $\{\tau_{b_i}(U)\}_{i=1}^{\infty}$ have multiplicity atmost $M_0.$
 
Now for each $m \geq M,$ let us define
\begin{align}
 \mathcal{S}_m := \{b_i : ||u_m \circ \tau_{b_i}||^2_{H^k_g(U)} > \frac{q}{2\beta_0 p}\}
\end{align}
Then as in \cite{SD} we can show that number of elements in $\mathcal{S}_m$ is uniformly bounded by a constant, say 
$\gamma_0,$ and
\begin{align*}
 \sum_{b_i \notin \mathcal{S}_m} \int_{\tau_{b_i}(V)} \left(e^{\beta_0 p u^2_m} - 1 \right) \ dv_g \leq \tilde C,
\end{align*}
where $\tilde C$ is independent of $m.$

Now assume $b_{i_m} \in \mathcal{S}_m,$ where $m \geq M.$ Then we see that
\begin{align}
 ||(u_m \circ \tau_{b_{i_m}}) - (u \circ \tau_{b_{i_m}})||^2_{k,g}
 = ||u_m - u||^2_{k,g} < \frac{1}{p}.
\end{align}
For simplicity let us write:
\begin{align*}
 v_m := ((u_m \circ \tau_{b_{i_m}}) - (u \circ \tau_{b_{i_m}})) \ \ \mbox{and} \ \ w_m := u \circ \tau_{b_{i_m}}
\end{align*}
Then we can write $e^{(u_m \circ \tau_{b_{i_m}})^2 }$ as
\begin{align*}
 e^{(u_m \circ \tau_{b_{i_m}})^2 } = e^{v_m^2} e^{2v_m w_m} e^{w^2_m}.
\end{align*}
Now 
\begin{align*}
 \int_{\tau_{b_{i_m}}(V)} \left(e^{\beta_0 p u^2_m} - 1\right) \ dv_g
 &\leq C \int_{\bn} e^{\beta_0 p (u_m \circ \tau_{b_{i_m}})^2}  \ dx \\
 &\leq C\left(\int_{\bn} e^{\beta_0 pq_1 v^2_m} \ dx\right)^{\frac{1}{q_1}} I_m,
 \end{align*}
where 
\begin{align*}
 I_m := \left(\int_{\bn}e^{2\beta_0 p q_2 v_m w_m } \ dx\right)^{\frac{1}{q_2}}
 \left(\int_{\bn} e^{2\beta_0 p q_3 w^2_m } \ dx\right)^{\frac{1}{q_3}},
\end{align*}
and $q_1, q_2, q_3$'s are chosen such that $\frac{1}{q_1} + \frac{1}{q_2} + \frac{1}{q_3} = 1, \ ||v_m||_{k,g} < \frac{1}{pq_1}$
for all $m \geq M.$ Now we note that for $1 \leq q < \infty,$
\begin{align} \label{estimateI}
 \int_{\bn} e^{qw^2_m} \ dx &\leq C + C \int_{\hn} \left(e^{qw_m^2} - 1\right) \ dv_g 
  C + C \int_{\hn} \left(e^{qu^2} - 1\right) \ dv_g
 \leq C(q).
\end{align}
Since $\sup_m ||v_m||_{k,g}$ is finite, using \eqref{estimateI} one can easily show that $ I_m $ 
can be bound by a positive constant $C_1$ independent of $m.$
Since number of elements in $\mathcal{S}_m$ is 
at most $\gamma_0$ we conclude that,
\begin{align}
 \sum_{b_i \in \mathcal{S}_m} \int_{\tau_{b_i}(V)} \left(e^{\beta_0 p u_m^2} - 1\right) \ dv_g \leq C_0
\end{align}
where $C_0$ is independent of $m.$ This proves the theorem.
\end{proof}

\section{Proof of Theorems: Part II}

\textbf{Proof of Theorem \ref{main pll 3} :} 
\begin{proof}
In the proof we will not distinguish between the original sequence and its subsequence, one can easily figure out from the context.
Also by standard argument, passing to a subsequence does not effect the main result.
We will also strictly follow the following notation to avoid confusions : for a subset $U$ of $\bn,$ we will write $||h||_{L^q(U, \ dx)}$ 
to denote $L^q$ norm of $h$ with respect to the Lebesgue measure, otherwise it is assumed to be with 
respect to the hyperbolic measure $dv_g$ . Choose $p$ satisfying the condition of the theorem. 

 We divide the proof into two steps : 

\textbf{Step 1:} There exists $R_0 >$ such that
\begin{align}
 \sup_m \int_{\bn \backslash B(0,R_0)} \Phi \left(\beta_0 p |u_m|^{\frac{N}{N - 2}}\right) \ dv_g < + \infty
\end{align}
\textbf{Proof :} Since $p < \left(1 - ||\Delta_g u||^{\frac{N}{2}}_{L^{\frac{N}{2}}(\bn)}\right)^{-\frac{2}{N - 2}},$ we
can find $\epsilon > 0$ small enough such that 
\begin{align} \label{estimate1 p}
 \frac{1}{p^{\frac{N - 2}{N}}} > \Big[1 - \Big(||\Delta_g u||_{L^{\frac{N}{2}}(\bn)} - 2 \epsilon \Big)^{\frac{N}{2}}
 \Big]^{\frac{2}{N}} + \epsilon.
\end{align}
By regularity of the measure we can find an open set $K$ such that $K$ is relatively compact in $\bn$ and 
satisfies,
\begin{align} \label{estimate2 p}
||\Delta_g u||_{L^{\frac{N}{2}}(K)} \geq ||\Delta_g u||_{L^{\frac{N}{2}}(\bn)} - \epsilon.
\end{align}
Therefore \eqref{estimate1 p} and \eqref{estimate2 p} together gives
\begin{align} \label{estimate p}
 \frac{1}{p^{\frac{N - 2}{N}}} > \Big[1 - \Big(||\Delta_g u||_{L^{\frac{N}{2}}(K)} -  \epsilon \Big)^{\frac{N}{2}}
 \Big]^{\frac{2}{N}} + \epsilon.
\end{align}
Again using $||\Delta_g u||_{L^{\frac{N}{2}}(K)} = 
\sup \ \left\{\int_{K} (\Delta_g u) \phi \ dv_g :\phi \in C^{\infty}_c(K), 
||\phi||_{L^{\frac{N}{N - 2}}(K)}\leq 1 \right\}, $
 we can find a $\phi_1 \in C^{\infty}_c(K)$ with $ ||\phi||_{L^{\frac{N}{N - 2}}(K)} \leq 1$
such that 
\begin{align} \label{estimate1 um}
 \int_{K} (\Delta_g u) \phi_1 \ dv_g \geq ||\Delta_g u||_{L^{\frac{N}{2}}(K)} - \frac{\epsilon}{2}.
\end{align}
Now \eqref{estimate1 um} together with the weak convergence gives: there exists a positive integer $m_0$ such that
\begin{align}
 ||\Delta_g u_m||_{L^{\frac{N}{2}}(K)} \geq ||\Delta_g u||_{L^{\frac{N}{2}}(K)} - \epsilon, \ \ \mbox{for all} \ m \geq m_0.
\end{align}
Therefore we have :
\begin{align} \label{estimate2 um}
 ||\Delta_g u_m||^{\frac{N}{2}}_{L^{\frac{N}{2}}(\bn \backslash K)} 
 &= 1 - ||\Delta_g u_m||^{\frac{N}{2}}_{L^{\frac{N}{2}}(K)} \notag \\
 &\leq 1 - \Big(||\Delta_g u||_{L^{\frac{N}{2}}(K)} - \epsilon \Big)^{\frac{N}{2}},
 \ \ \mbox{for all} \ m \geq m_0.
 \end{align}
It is clear from \eqref{estimate2 um} that , for any $r \in (0,1)$ with $K \subset B(0,r)$ we have 
\begin{align*}
 ||\Delta_g u_m||^{\frac{N}{2}}_{L^{\frac{N}{2}}(\bn \backslash B(0,r))}
 \leq 1 - \Big(||\Delta_g u||_{L^{\frac{N}{2}}(K)} - \epsilon \Big)^{\frac{N}{2}},
 \ \ \mbox{for all} \ m \geq m_0.
\end{align*}

Let $C_0$ be the constant as appeared in Lemma \ref{test function}, and choose $\eta > 0$ such that 
\begin{align}
 8(N - 2) C_0 \eta < \epsilon.
\end{align}
Choose $r_0$ such that $B(0,r_0)$ contains $K$ and the following  holds:
\begin{align}
 ||u||_{L^{\frac{N}{2}}(\bn \backslash B(0,r_0))} + ||\nabla_g u||_{L^{\frac{N}{2}}(\bn \backslash B(0,r_0))}
 < \frac{\eta}{2^{\frac{2}{N}}}
\end{align}
Let us take $\delta_0 = \frac{1 - r_0}{2},$ then by Lemma \ref{test function}, there exists $\phi_0$ such that
$\phi_0$ is of class $C^{\infty}, 0 \leq \phi_0 \leq 1, \phi_0 \equiv 1$ on $B(0,r_0),$ supp$(\phi_0) \subset 
B(0, r_0 + \delta_0)$ and there holds :
\begin{align}
 \delta_0||\nabla \phi_0||_{L^{\infty}} + \delta_0^2 ||\Delta \phi_0||_{L^{\infty}} \leq C_0.
\end{align}
Now define $\psi_0 = (1 - \phi_0).$ We will now estimate $||\Delta_g (\psi_0 u_m)||_{L^{\frac{N}{2}}(\bn)}.$ 
 
We note that :
\begin{align} \label{estimate  psium}
 ||\Delta_g (\psi_0 u_m)||_{L^{\frac{N}{2}}(\bn)} &\leq
 ||\psi_0 \Delta_g u_m||_{L^{\frac{N}{2}}(\bn)} + 2|| \langle \nabla_g \psi_0 , 
 \nabla_g u_m \rangle_g||_{L^{\frac{N}{2}}(\bn)} \notag \\
 &+ ||u_m \Delta_g \psi_0||_{L^{\frac{N}{2}}(\bn)}
\end{align}
We will now estimate each term on the right hand side of \eqref{estimate  psium}. First we estimate the second term.
\begin{align}
 || \langle \nabla_g \psi_0 , \nabla_g u_m \rangle_g||^{\frac{N}{2}}_{L^{\frac{N}{2}}(\bn)} 
 &\leq \int_{\bn} |\nabla_g \psi_0|^{\frac{N}{2}}_g |\nabla_g u_m|^{\frac{N}{2}}_g \ dv_g, \notag \\
 &\leq (1 - r_0)^{\frac{N}{2}} ||\nabla \psi_0||^{\frac{N}{2}}_{L^{\infty}} 
 \int_{\{r_0 < |x| < r_0 + \delta_0 \}} |\nabla_g u_m|^{\frac{N}{2}}_g \ dv_g \notag \\
 &\leq 2^{\frac{N}{2}} \delta_0^{\frac{N}{2}}||\nabla \phi_0||^{\frac{N}{2}}_{L^{\infty}}
 \int_{\{r_0 < |x| < r_0 + \delta_0 \}} |\nabla_g u_m|^{\frac{N}{2}}_g \ dv_g \notag \\
 &\leq (2C_0)^{\frac{N}{2}} \int_{\{r_0 < |x| < r_0 + \delta_0 \}} |\nabla_g u_m|^{\frac{N}{2}}_g \ dv_g.
\end{align} 
Now using compact embedding we get : there exists $m_1$ such that for all $m \geq m_1,$
\begin{align}
 || \langle \nabla_g \psi_0 , \nabla_g u_m \rangle_g||^{\frac{N}{2}}_{L^{\frac{N}{2}}(\bn)}
 &\leq (2C_0)^{\frac{N}{2}} \left(\int_{\{r_0 < |x| < r_0 + \delta_0 \}} |\nabla_g u|^{\frac{N}{2}}_g \ dv_g
 + \frac{\eta^{\frac{N}{2}}}{2} \right) \notag \\
&\leq (2 C_0 \eta)^{\frac{N}{2}}. 
\end{align}
Now we will estimate the last term in \eqref{estimate  psium}. Note that on $\{r_0 < |x| < r_0 + \delta_0 \},$
\begin{align}
 |\Delta_g \psi_0| &\leq (1 - r_0)^2||\Delta \psi_0||_{L^{\infty}} + (N - 2)(1 - r_0)||\nabla \psi_0||_{L^{\infty}}, \notag \\
 &\leq 4(N - 2) \left(\delta_0 ||\nabla \psi_0||_{L^{\infty}} + \delta^2_0 ||\Delta \psi_0||_{L^{\infty}} \right), \notag \\
 &\leq 4(N - 2)C_0.
\end{align}
Therefore using compact embedding we get : there exists $m_2$ such that for all $m \geq m_2,$ 
\begin{align}
 ||u_m \Delta_g \psi_0||^{\frac{N}{2}}_{L^{\frac{N}{2}}(\bn)}
 &\leq (4(N - 2)C_0)^{\frac{N}{2}} \left(\int_{\{r_0 < |x| < r_0 + \delta_0 \}} |u|^{\frac{N}{2}} \ dv_g 
 + \frac{\eta^{\frac{N}{2}}}{2}\right) \notag \\
 &\leq [4(N - 2) C_0 \eta]^{\frac{N}{2}}.
\end{align}
Since $\psi_0 \equiv 0$ on $B(0,r_0)$ and $r_0$ was choosen so that $B(0,r_0)$ contains $K$ we have,
\begin{align}
 ||\Delta_g (\psi_0 u_m)||_{L^{\frac{N}{2}}(\bn)} &\leq ||\Delta_g u_m||_{L^{\frac{N}{2}}(\bn \backslash K )} 
 + 8(N - 2)C_0 \eta  \notag \\
 &\leq \left[1 - \Big(||\Delta_g u||_{L^{\frac{N}{2}}(K)} - \epsilon \Big)^{\frac{N}{2}}\right]^{\frac{2}{N}}
 + \epsilon \notag \\
 &\leq \frac{1}{p^{\frac{N - 2}{2}}}.
\end{align}
Therefore using Adams inequality on the hyperbolic space (\cite{Fontana15}) we get 
\begin{align} \label{conclude step 1}
 \sup_{m \geq M} \int_{\bn} \Phi_{2,N} \left(\beta_0 p |\psi_0 u_m|^{\frac{N}{N - 2}}\right) \ dv_g < + \infty,
\end{align} 
where $M = \max \{m_0 , m_1, m_2\}.$ Since $\psi_0 \equiv 1$ for $|x| \geq r_0 + \delta_0,$ choosing $R_0 = (r_0 + \delta_0)$
we get from \eqref{conclude step 1},

\begin{align}
 \sup_m \int_{\bn \backslash B(0,R_0)} \Phi_{2,N} \left(\beta_0 p | u_m|^{\frac{N}{N - 2}}\right) \ dv_g < + \infty.
\end{align}
This completes the proof of step 1. 

\vspace{0.3 cm}

\textbf{Step 2 :} For any $R \in (0,1),$ there holds
\begin{align}
 \sup_m \int_{B(0,R)} e^{\beta_0 p |u_m|^{\frac{N}{N - 2}}} \ dx < +\infty.
\end{align}
\textbf{Proof :} For this step, without loss of generality we can assume that $u_m \in \Cbn$ for all $m.$
We can write $\Delta_g u_m$ as
\begin{align*}
 \Delta_g u_m = \left(\frac{1 - |x|^2}{2}\right)^2 f_m,
\end{align*}
where $f_m = \Delta u_m + (N - 2) \left(\frac{2}{1 - |x|^2}\right) \langle x, \nabla u_m \rangle.$
Let $v_m$ be the solution of the equation :
\begin{align}
 - \Delta v_m &= |f_m|^{*} \ \ \ \mbox{in} \ \bn , \notag \\
 v_m &= 0 \ \ \ \ \ \ \ \  \mbox{on} \ \partial \bn . 
\end{align}

Let $\Delta_g u = \left(\frac{1 - |x|^2}{2}\right)^2 f ,$ then it follows from weak convergence of $u_m$ that
\begin{align*}
 f_m \rightharpoonup f \ \ \ \mbox{in} \ \ L^{\frac{N}{2}}(\bn, dx).
\end{align*}
Since $\{|f_m|\}_m$ is a bounded sequence in  $L^{\frac{N}{2}}(\bn, dx),$ it follows that upto a subsequence
which is still denoted by $|f_m|$ converges to some
$\tilde f$ weakly in $L^{\frac{N}{2}}(\bn, dx).$ \\
\textbf{Claim:} $||\tilde f||_{L^{\frac{N}{2}}(\bn, dx)} \geq ||f||_{L^{\frac{N}{2}}(\bn, dx)}.$ \\
Proof of the claim: Let $\phi \in \Cbn,$ with $\phi \geq 0,$ then
\begin{align*}
 \int_{\bn} (\tilde f - f) \phi \ dx = \lim_{m} \int_{\bn} (|f_m| - f_m) \phi \ dx \geq 0.
\end{align*}
Similarly,
\begin{align*}
 \int_{\bn} (\tilde f + f) \phi \ dx = \lim_{m} \int_{\bn} (|f_m| + f_m) \phi \ dx \geq 0.
\end{align*}
This proves that $\tilde f \geq |f|$ a.e in $\bn,$ and hence
$||\tilde f||_{L^{\frac{N}{2}}(\bn, dx)} \geq ||f||_{L^{\frac{N}{2}}(\bn, dx)}.$  This proves the claim.

Now applying Lemma 2 of \cite{DoMac} we get, $|f_m|^{\#}$ converges to some $g$ point wise a.e in $(0, |\bn|)$ with
\begin{align*}
||g||_{L^{\frac{N}{2}}(0, |\bn|)} \geq ||\tilde f^{\#}||_{L^{\frac{N}{2}}(0, |\bn|)}  \geq ||f||_{L^{\frac{N}{2}}(\bn, dx)} .
\end{align*}
Let us define $g_0(x) = g(\sigma_N |x|^N)$ for $x \in \bn,$ and let $v$ be the solution of the equation
\begin{align}
 - \Delta v &= g_0 \ \ \ \mbox{in} \ \bn  , \notag \\
 v &= 0 \ \ \ \  \mbox{on} \ \partial \bn  .
\end{align}
Then it follows that,
\begin{align*}
 &\Delta v_m \rightarrow \Delta v  \ \  \ \ \mbox{point wise a.e in} \ \bn \ \mbox{and} \\
 &||\Delta v||_{L^{\frac{N}{2}}(\bn, dx)} = ||g_0||_{L^{\frac{N}{2}}(\bn, dx)} \geq
 ||f||_{L^{\frac{N}{2}}(\bn, dx)} \geq	||\Delta_g u||_{L^{\frac{N}{2}}(\bn)}.
\end{align*}
Applying Brezis-Leib lemma we get, for $m$ sufficiently large
\begin{align*}
 ||\Delta (v_m - v)||^{\frac{N}{2}}_{L^{\frac{N}{2}}(\bn, dx)} &= 1 - ||\Delta v||^{\frac{N}{2}}_{L^{\frac{N}{2}}(\bn, dx)} + o_m(1) ,\\
 &\leq 1 - ||\Delta_g u||^{\frac{N}{2}}_{L^{\frac{N}{2}}(\bn)} + o_m(1).
\end{align*}
Choose $p_0$ such that 
\begin{align}
 p < p_0 < \left(1 - ||\Delta_g u||^{\frac{N}{2}}_{L^{\frac{N}{2}}(\bn)}\right)^{-\frac{2}{N - 2}}.
\end{align}
Then for sufficiently large $m$ we get 
\begin{align}
 ||\Delta(v_m - v)||_{L^{\frac{N}{2}}(\bn, dx)} < \frac{1}{p_0^{\frac{N - 2}{N}}}
\end{align}
and therefore we get,
\begin{align} \label{euclidean pll}
 \sup_m \int_{\bn} \left(e^{\beta_0 p_0 |v_m|^{\frac{N}{N - 2}}} - 1 \right) \ dx< + \infty.
\end{align}
Now applying Lemma \ref{main lemma} we get that there exists a $\tau > 0$ such that :
\begin{align*}
 \int_{B(0,R)} \left(e^{\beta_0 p |u_m|^{\frac{N}{N - 2}}} - 1 \right) \ dx ,
 &\leq \int_{B(0,R)} \left(e^{\beta_0 p |v_m + \tau|^{\frac{N}{N - 2}}} - 1 \right) \ dx , \\
 &\leq \int_{\bn} \left(e^{\beta_0 p |v_m + \tau|^{\frac{N}{N - 2}}} - 1 \right) \ dx , \\
 &\leq C\int_{\bn} \left(e^{\beta_0 p_0 |v_m|^{\frac{N}{N - 2}}} - 1 \right) \ dx.
\end{align*} 
This proves that,
\begin{align*}
 \sup_m \int_{B(0,R)} \left(e^{\beta_0 p |u_m|^{\frac{N}{N - 2}}} - 1 \right) \ dx < + \infty,
\end{align*}
and hence
\begin{align*}
 \sup_m \int_{B(0,R)} \left(e^{\beta_0 p |u_m|^{\frac{N}{N - 2}}} - 1 \right) \ dv_g < + \infty.
\end{align*}
Therefore step 1 and step 2 together completes the proof of the theorem.
\end{proof}

\pagebreak

\textbf{Proof of Theorem \ref{main pll1} :}

\begin{proof} As before we will not distinguish between the original sequence and its subsequence.
 Choose $p$ satisfying the condition of the theorem.
 We divided the proof into two steps : \\
 \textbf{Step 1:} There exists $R_0 > 0,$ such that,
 \begin{align}
  \sup_m \int_{\bn \backslash B(0,R_0)} \Phi_{1,N} (\alpha_N p |u_m|^{\frac{N}{N - 1}}) \ dv_g < + \infty.
 \end{align}
The proof of step 1 goes exactly in the same line as in the proof of theorem \ref{main pll 3}, so we omit the proof.

\textbf{Step 2:} For any $R \in (0,1)$
\begin{align}
  \sup_m \int_{B(0,R)} \Phi_{1,N} (\alpha_N p |u_m|^{\frac{N}{N - 1}}) \ dx \leq C(R) < + \infty.
\end{align}
 This is the Euclidean P.L.Lions lemma and follows from \cite{P.L.L} (see also \cite{CCH}).  

Hence, we have
 \begin{align*}
  \sup_m \int_{B(0,R)} \Phi_{1,N} (\alpha_N p |u_m|^{\frac{N}{N - 1}}) \ dv_g 
  &\leq C \sup_m \int_{\bn} e^{\alpha_N p |u_m|^{\frac{N}{N - 1}}} \ dx \\
  &< + \infty.
 \end{align*}
Therefore step 1 and step 2 combined proves the theorem.
\end{proof}

\end{document}